\documentclass{amsart}
\usepackage[ansinew]{inputenc}
\usepackage{latexsym}
\usepackage{amsmath,amsfonts,amssymb,amsopn}

\newenvironment{proofmainthm}{\smallskip \noindent {\it Proof of Theorem \ref{mainthm}.}}{\hfill \rule{2mm}{2mm}\hspace{1in}\smallskip}
\newtheorem{teo}{Theorem}[section]
\newtheorem{lema}{Lemma}[section]
\newtheorem{prop}{Proposition}[section]
\newtheorem{defi}{Definition}[section]
\newtheorem{obs}{Remark}[section]

\newcommand{\loc}{_{loc}}
\newcommand{\vare}{\varepsilon}
\newcommand{\real}{\mathbb{R}}
\title[Weak stability of Lagrangian solutions]{Weak stability of Lagrangian solutions to the semigeostrophic equations}

\author[]{Josiane C. O.  Faria}
\address{Departamento de Matem\'{a}tica\\IMECC, Cx. Postal 6065\\Universidade Estadual de Campinas - UNICAMP\\ Campinas, SP 13083-970, Brazil} 
\email{josianecristina@gmail.com}

\author[]{Milton C. Lopes Filho}
\address{Departamento de Matem\'{a}tica\\IMECC, Cx. Postal 6065\\Universidade Estadual de Campinas - UNICAMP\\ Campinas, SP 13083-970, Brazil}
\email{mlopes@ime.unicamp.br}

\author[Helena J. Nussenzveig Lopes {\em et al.}]{Helena J. Nussenzveig Lopes}
\address{Departamento de Matem\'{a}tica \\IMECC, Cx. Postal 6065 \\Universidade Estadual de Campinas - UNICAMP\\Campinas SP 13083-970, Brazil}
\email{hlopes@ime.unicamp.br}

\date{}
\begin{document}

\maketitle

\begin{abstract}
In \cite{culefel}, Cullen and Feldman proved existence of Lagrangian
solutions for the semigeostrophic system in physical variables with
initial potential vorticity in $L^p$, $p>1$. Here, we show that a 
subsequence of the Lagrangian solutions corresponding to a strongly convergent 
sequence of initial potential vorticities in $L^1$ converges strongly in 
$L^q$, $q<\infty$, to a Lagrangian solution, in particular extending the existence result 
of Cullen and Feldman to the case $p=1$. We also present a counterexample for Lagrangian 
solutions corresponding to a sequence of initial potential vorticities converging 
in $\mathcal{BM}$.  The analytical tools used include techniques from optimal 
transportation, Ambrosio's results on transport by $BV$ vector fields, 
and Orlicz spaces.
\end{abstract}

\section{Semigeostrophic equations in physical and in dual variables}

Semigeostrophic equations are simplified models for large-scale
geophysical flows. These systems were introduced by Hoskins in \cite{hoskins}, 
as part of a family of models for geophysical flows under approximate geostrophic 
balance, i.e. where Coriolis forces and horizontal gradients of pressure nearly
balance. We refer the reader to \cite{cul1} for a thorough account of 
semigeostrophy from the physical point of view.

In the present work, we are concerned with two versions of the semigeostrophic
system -- the incompressible 3D system and the shallow water system, both in a
bounded domain in $\mathbb{R}^3$ with constant Coriolis force. We will first focus the
discussion on the incompressible case, leaving the shallow water system to Section 4.
The incompressible semigeostrophic equation, or SG equation, has a rich mathematical structure, 
closely related with optimal transport theory. Written in {\it dual variables}, it can be 
interpreted as a fully nonlinear active scalar equation, where the transported scalar, called {\it
potential vorticity}, generates the transporting velocity by means of a 
Monge-Amp\`{e}re equation. Several results are available for the semigeostrophic 
system in dual form, beginning with the work of J.-D. Benamou and Y. Brenier, \cite{ben}, 
on existence of weak solutions, and including \cite{culegan,culemar,helena,mo04}.
In \cite{loeper}, G. Loeper proved existence of weak solutions for 
potential vorticities in the space of Radon measures. 
Obtaining a solution in physical variables from weak solutions of the
dual form is both delicate and physically relevant. For potential vorticities in
$L^p$, $p>1$, this problem was partially solved by M. Cullen and M. Feldman in \cite{culefel},
with the introduction of Lagrangian solutions to the system in physical
variables.  

This article´s main concern is the weak stability of Lagrangian solutions with respect to 
perturbations in the initial potential vorticity, complementing 
the work of Cullen and Feldman. Our main result is the weak 
compactness in $L^p$, for any $p < \infty$, of sequences of Lagrangian solutions, obtained 
from sequences of initial potential vorticities converging strongly in $L^1$. In addition, 
we also include analysis of the compressible shallow water case, as in \cite{culefel}. Finally, we
present a counterexample for the extension of our main result to initial potential vorticities 
in the space of measures. 
 
Let $\Omega\subset\mathbb{R}^3$ be open 
and bounded. The 3D incompressible semigeostrophic equations in physical variables are
a system of four equations with three components of velocity $u = (u_1,u_2,u_3)$ and
the pressure $p$ as unknown. Before we write down the system, we introduce the 
{\it geostrophic velocity} $v^g = (v^g_1,v^g_2,0)$ as $v^g = (-\partial_2 p, \partial_1 p , 0)$
and the density $\rho$ by assuming vertical hydrostatic balance $\rho = -\partial_3 p$.
The semigeostrophic equations have the form

\begin{equation} \left\{
\begin{array}{l}
D_t(v^g_1,v^g_2)+(v^g_2,-v^g_1)=(u_2,-u_1)\\
D_t\rho=0\\
\mbox{div } u=0,\\
\end{array} \right. \label{se}
\end{equation}
where $D_t = \partial_t + u \cdot \nabla$ is the material derivative.
The unknowns are functions of $(x,t)\in\Omega\times[0,T)$ with initial
and boundary data given by
\begin{equation}
\begin{array}{l}
u \cdot\nu=0\;\mbox{ on }\;\partial\Omega\times[0,T)\\
p(x,0)=p_0(x)\;\mbox{ in }\;\Omega, \end{array}
\end{equation}
where $\nu$ is the unit outer normal to $\partial \Omega$.
Taken together, these four equations make up a rather odd-looking system of PDEs. We have 
three transport equations for the components of $\nabla p$, in which $u$ enters only algebraically,
together with the divergence-free condition. That this system turns out to be solvable
only becomes apparent after a change of variables, which expresses its dual formulation.

In order to present the dual variable formulation of this system, we consider
the modified pressure $P$, given by

$$P(x,t)=p(x,t)+\frac{1}{2}(x^2_1+x^2_2)$$
and we rewrite the semigeostrophic system as
\begin{equation} \left\{
\begin{array}{l}
D_tX=J(X-x)\\
\mbox{div } u=0\\
X=\nabla P\\
u \cdot \nu=0\;\mbox{ on }\;\partial\Omega\times[0,T)\\
P(x,0)=P_0(x)\;\mbox{ in }\;\Omega, \end{array} \right. \label{la}
\end{equation}
where $$\displaystyle{J=\left(\begin{array}{rrrr} 0&-1&&0\\1&0&&0\\0&0&&0
\end{array}\right)}.$$

The dual formulation is obtained by switching dependent and independent variables
in the system above. More precisely, we assume that $P$ and $u$ are solutions of the system above, 
and, in addition, $P(\cdot,t)$ is convex and smooth for all $t$. A consequence of the 
convexity of $P$ is that $\nabla P(\cdot,t)$ is a diffeomorphism between $\Omega$ and some subset 
of $\real^3$. We introduce $\alpha \equiv \nabla P \sharp \chi_{\Omega}$, where the sharp indicates 
measure pushforward and $\chi_{\Omega}$ denotes the Lebesgue measure in $\Omega$.
To be precise, if $\Omega_1$ and $\Omega_2$ are subsets of $\real^n$, $\mu$ is a measure on $\Omega_1$ and
$\nu$ is a measure on $\Omega_2$ and $X$ maps $\Omega_1$ to a subset of $\Omega_2$ then the notation
$X \sharp \mu = \nu$ means that:
\[ \int_{\Omega_2} f(y) d\nu = \int_{\Omega_2} f(X(x)) d\mu, \]
for any $f \in C^0(\Omega_2)$.

The measure $\alpha$ is called {\it potential vorticity}.
We also introduce $P^{\ast}$ the Legendre transform of $P$, i.e. 
\[P^{\ast}(X,t) \equiv \displaystyle{\sup_{x\in\Omega}\{x \cdot X-P(x,t)\}}.\]
The potential vorticity $\alpha$ satisfies the following system of equations:

\begin{equation} \left\{
\begin{array}{ll}
\partial_t\alpha+\nabla \cdot (U\alpha)=0&\mathbb{R}^3\times[0,T)\\
U(X,t)=J[X-\nabla P^{\ast}(X,t)]&\\
\nabla P(\cdot,t) \sharp \chi_{\Omega} = \alpha(\cdot,t)\\
\alpha(X,0)=\alpha_0(X)&\mbox{ a.e. }\;X\in \mathbb{R}^3.
\end{array} \right. \label{du}
\end{equation}

From the definition of the pushforward measure, we can see that
the statement $\nabla P(\cdot,t) \sharp \chi_{\Omega} = \alpha(\cdot,t)$ amounts to a weak
form of the equation $\mbox{det }(D^2 P^{\ast}) = \alpha$, with the condition that the image of 
$\nabla P^{\ast}$ is $\Omega$. This observation shows that \eqref{du} is an active scalar transport equation,
where the transporting velocity is determined from the transported scalar by means of a Monge-Amp\`{e}re equation.
The derivation of the dual system from the physical system is a standard calculation, and it can be found, 
for example, in \cite{ben}. The key hypothesis for the validity of this derivation is the convexity of $P$,
something which is preserved under semigeostrophic evolution, see \cite{ben}. 

The following result concerns existence of weak solutions for the semigeostrophic
system in dual variables. Let $\Omega$ be a bounded domain in $\real^3$, which we assume
is contained in the ball of radius $S>0$, centered at the origin. Let $T>0$. 

We fix the bounded smooth domain $\Omega$ in physical space, the radius $S>0$ and the time 
horizon $T>0$ throughout the remainder of this paper.

We require the notation and terminology associated to Orlicz spaces in the statement below. We will briefly account 
for the basic theory of Orlicz spaces in Section 2.

\begin{teo}\cite{ben,culefel,helena} \label{existheoMaut}
Let $P_0=P_0(x)$ be a bounded, convex function in $\Omega$,
and let $\alpha_0:=DP_0\#\chi_{\Omega}$. Suppose that $\alpha_0 \in L^q(\mathbb{R}^3)$
for some $q\geq 1$, and it is compactly supported. Let $R_0$ be such that the support of $\alpha_0$ 
is contained in the ball $B(0,R_0)$ and set $R(T) = R_0 e^T + (e^T - 1)S$.
Then, for any $t>0$, there exist functions
$\alpha = \alpha(X,t) \in L^{\infty}([0,T),L^q(\mathbb{R}^3))$, $P=P(X,t) \in L^{\infty}([0,T),W^{1,\infty}(\Omega))$,  
and an $N$-function $A$ such that
\begin{enumerate}
\item[(i)] 
  \begin{equation}
  supp(\alpha(\cdot,t))\subset B(0,R(T));\;\forall t\in[0,T),
  \end{equation}
\item[(ii)] For each $0\leq t < T$, $P(\cdot,t)$ is convex, and $\alpha$, $P$ satisfy
$$\alpha \in  C([0,T),L_A(B(0,R(T))))\; \mbox{ and } \; P \in C([0,T),W^{1,r}(\Omega)),$$
for any real number $ r \in  [1,\infty)$,
where $L_A$ is the Orlicz space associated with the $N$-function $A$.
\item[(iii)] $P^{\ast}(\cdot,t)$ is convex, pointwise in time, locally bounded in space-time
    and
\[\nabla P^{\ast}\in L^{\infty}([0,T),E_{A^{\ast}}(B(0,R(T))))\cap
C([0,T),L^r(B(0,R(T))))\]
for any $r\in[1,\infty)$, where $E_{A^{\ast}}$ is the dual of $L_A$. Moreover,
    \begin{equation}
    \Vert\nabla P^{\ast}(\cdot ,t)\Vert_{L^{\infty}(\mathbb{R}^3)}\leq
    S\quad\forall t\in [0,T),
    \end{equation}
    \item[(iv)] $(\alpha,P,P^{\ast})$ satisfy (\ref{du}), where the evolution equation and the initial data for
$\alpha$ are understood in the weak sense, i.e., for each $\phi \in C^1_c(\mathbb{R}^3\times [0,T))$
\begin{equation}
\begin{array}{l}
\displaystyle{\int_{\mathbb{R}^3\times[0,T)}[\partial_t\phi(X,t)+U(X,t) \cdot \nabla\phi(X,t)]\alpha(X,t)dXdt+}\\
\;\;\;\;\;\;\;\;\;\;\;\;\;\;\;\;\;\;\;\;\;\;\;\displaystyle{+\int_{\mathbb{R}^3}\alpha_0(X)\phi(X,0)dX=0.}
\end{array}
\end{equation}
\end{enumerate}
\end{teo}
\vskip10pt

One would like to find solutions for the SG system in physical space, which means, solutions to 
\eqref{la}. Our point of departure is the work of Cullen and Feldman in \cite{culefel}. In that
article, Cullen and Feldman pointed out that:  
\begin{itemize}
\item concerning the  Eulerian form of system \eqref{la} -- a distributional formulation of this system
requires making sense of products of components of $u$ with first derivatives of $P$;
\item the formal expression for $u$ is given by $$u(x,t)=\partial_t\nabla P^{\ast}(\nabla
P(x,t),t)+D^2P^{\ast}(\nabla P(x,t),t)[J(\nabla P(x,t)-x)];$$
\item the known regularity for solutions of the dual problem has $P^{\ast}$
Lipschitz continuous.
\end{itemize}

Clearly, making sense of the physical velocity $u$ is complicated, given that $D^2 P^{\ast}$ is a measure while $\nabla P$ is 
only bounded, not to mention making sense of the product $u\cdot\nabla P$. As a consequence, seeking Eulerian solutions in physical variables is 
a difficult problem. This was the motivation for the introduction of the notion of Lagrangian solutions.  
  
 Let $P_0\in
W^{1,\infty}(\Omega)$ be a convex function and $r\in [1,\infty)$. Let
$P:\Omega\times[0,T)\rightarrow \mathbb{R}$ satisfy
\begin{eqnarray}
P \in L^{\infty}([0,T);W^{1,\infty}(\Omega))\cap
C([0,T),W^{1,r}(\Omega)),\label{2.33}\\
P(\cdot ,t)\;\mbox{is convex in} \;\Omega,\forall t\in [0,T).
\end{eqnarray}
Let $F:\Omega\times[0,T)\rightarrow\Omega$ be a Borelian map such that 
\begin{equation}
F\in C([0,T),L^r(\Omega)).\label{2.35} 
\end{equation}

\begin{defi}{\bf{(Lagrangian Solutions)}}  The pair
$(P,F)$ is called a Lagrangian solution of (\ref{la}) in
$\Omega\times [0,T)$ if
\begin{enumerate}
    \item[(i)] $F(x,0)=x, P(x,0)=P_0(x)$ for almost all $x\in\Omega$,
    \item[(ii)] for all $0\leq t < T$ the mapping
    $F_t=F(\cdot,t):\Omega\rightarrow\Omega$ preserves   Lebesgue measure, i.e.
    $F_t\#\chi_{\Omega}=\chi_{\Omega}$,
    \item[(iii)] There exists a Borelian map
    $F^{\ast}:\Omega\times[0,T)\rightarrow\Omega$ such that, for all
    $t\in (0,T)$, the map
    $F^{\ast}_t=F^{\ast}(\cdot,t)=\Omega\rightarrow\Omega$ preserves  Lebesgue measure,
    (i.e. $F^{\ast}_t\#\chi_{\Omega}=\chi_{\Omega}$), and satisfies
    $F_t\circ F^{\ast}_t(x)=x$ and $F^{\ast}_t\circ F_t(x)=x$ for almost all
    $x\in\Omega$,
    \item[(iv)] The function
    \begin{equation}
    Z=Z(x,t)=\nabla P(F_t(x),t)\label{2.36}
    \end{equation}
    is a weak solution of
    \begin{equation}
    \begin{array}{ll}
    \partial_tZ(x,t)=J[Z(x,t)-F(x,t)],& \mbox{em}\quad\Omega\times[0,T)\\
    Z(x,0)=\nabla P_0(x),&\mbox{sobre}\quad\Omega,
    \end{array}\label{2.37}\end{equation}
    in the following sense: for any $\varphi\in
    C^1_c(\Omega\times[0,T)),$
    \begin{equation}
    \begin{array}{l}
    \displaystyle{\int_{\Omega\times[0,T)}[Z(x,t) \cdot \partial_t\varphi(x,t)+J(Z(x,t)-F(x,t))\cdot \varphi(x,t)]dxdt}+\\
    \;\;\;\;\;\;\;\;\;\;\;\;\;\;\;\;\;\;\;\;\;\;\;\;\displaystyle{+\int_{\Omega}\nabla
    P_0(x)\cdot \varphi(x,0)dx=0.}
    \end{array}\label{2.38}
    \end{equation}
\end{enumerate}
\end{defi}
Given a Lagrangian solution $(P,F)$, the map $F(\cdot, t)$ is called a Lagrangian flow in physical space, for each $t \in [0,T)$.

Next we give the precise statement of existence of Lagrangian solutions.   
 
\begin{teo} \label{theo}
Let $P_0$ be convex and bounded in $B(0,S)$. Assume that
$P_0$ satisfies
\begin{equation}
DP_0\#\chi_{\Omega}\in L^q(\mathbb{R}^3)
\end{equation}
for some $q\geq1$, and that $DP_0\#\chi_{\Omega}$ is compactly supported. 
Then there exists a Lagrangian solution $(P,F)$ of (\ref{la})
in $\Omega\times[0,T)$, for which
(\ref{2.33})--(\ref{2.35}) are satisfied for any
$r\in[1,\infty)$. Moreover, the function $Z=Z(x,t)$, defined by
(\ref{2.36}), satisfies $Z(x,\cdot)\in W^{1,\infty}([0,T))$ for almost all
 $x\in\Omega$ and (\ref{2.37}) is also satisfied in the following sense:
\begin{equation}
\begin{array}{ll}
\partial_tZ(x,t)=J(Z(x,t)-F(x,t)),&\mbox{ for }\;(x,t)\in\Omega\times(0,T),\;\mathcal{L}^4-a.e.\\
Z(x,0)=\nabla P_0(x),&\mbox{ for }\;x\in\Omega,\;\mathcal{L}^3-a.e.
\end{array}
\end{equation}
\end{teo}

The case $q>1$ is the main result in \cite{culefel}. It was observed by one of the authors 
of the present paper that the result in \cite{culefel} can be extended to $q=1$ using the
technique from \cite{helena}. The proof of the case $q=1$ of Theorem \ref{theo} is embedded in 
the current work.

Let us briefly examine the construction underlying the proof of Theorem \ref{theo}.
Under the hypotheses of Theorem 1.2, from the solution of the dual problem, and using
the Polar Factorization Theorem (see \cite{bre}), one obtains $P$. Given $P$, it can be shown that the 
following expression gives rise to a Lagrangian flow in physical space:
\begin{equation} \label{flu}
F = F(x,t) = \nabla P^{\ast}_t\circ\Phi_t\circ\nabla P_0(x),
\end{equation}
where, for each $t$, $\Phi_t$ is the Lagrangian flow in dual space, obtained using
Ambrosio's theorem as follows.

Consider the transport equation
$$\partial_t\alpha+U.\nabla \alpha=0,$$
which is equivalent to the first equation  in (\ref{du}), since $\mbox{div }
U=0$. From the regularity of $P^{\ast}$  we have
$$U\in L^{\infty}_{loc}(\mathbb{R}^3\times[0,T)),\qquad U\in
L^{\infty}([0,T),BV_{\loc})$$

One  uses Ambrosio's theorem to obtain the Lagrangian flow
associated with the transport equation above. To do so one  must
modify the velocity $U$ near infinity without affecting the solution $\alpha$; this  can be achieved since $\alpha$ has compact support in 
$\mathbb{R}^3\times[0,T)$.  
We have 
\[\mbox{supp } \alpha\subset\overline{B(0,R(T))}\times[0,T],\]
where $R(T)$ was introduced in Theorem \ref{existheoMaut} (i).

We introduce a modified velocity $\tilde U$ as follows: choose $\varrho\in
C^{\infty}(\mathbb{R})$ such that
\begin{equation}
\varrho\equiv1\;\mbox{ in }\;\{\vert
s\vert<R(T)\},\quad\varrho\equiv0\;\mbox{ in }\;\{\vert
s\vert>R(T)+1\},\quad0\leq\varrho\leq1\;\mbox{ in }\;\mathbb{R},\label{2.46}
\end{equation}
and define, for $X\in\mathbb{R}^3$,
\begin{equation}
H(X)=(\varrho(\vert X_1\vert)X_1,\varrho(\vert
X_2\vert)X_2,\varrho(\vert X_3\vert)X_3).\label{2.47}
\end{equation}

The modified velocity $\tilde U$ is then given by
\begin{equation}
\tilde U(X,t)=J[H(X)-\nabla P^{\ast}(X,t)],
\label{2.48}\end{equation} and, therefore $\tilde U$ satisfies
\begin{equation}
\begin{array}{l}
\displaystyle{\tilde U\in L^{\infty}(\mathbb{R}^3\times [0,T))}\\
\displaystyle{\tilde U\in L^{\infty}([0,T),BV_{\loc})}\\
\displaystyle{\mbox{ div }\tilde
U(\cdot,t)=0}\;\mbox{ in }\;\mathbb{R}^3,\;\mbox{ for all }\;t\in[0,T).
\end{array} \end{equation}
Furthermore,
\begin{equation}
U=\tilde U\;\mbox{ in }\;B(0,R(T)),\end{equation}\begin{equation}
\Vert\tilde U(\cdot,t)\Vert_{L^{\infty}(\mathbb{R}^3)}\leq
S+R(T)+1\;\mbox{ for all }\;t\in[0,T).
\end{equation}
 
Therefore, using the results of L. Ambrosio in \cite{amb}, we have

\begin{prop} \label{ambrosio}
Let $\tilde U$ be given by (\ref{2.48}). There exists
a  unique locally bounded and Borel measurable mapping
$\Phi:\mathbb{R}^3\times [0,T)\rightarrow\mathbb{R}^3$ satisfying:
\begin{enumerate}
    \item[(i)] $\Phi(X,\cdot)\in W^{1,\infty}([0,T))$ for almost all
    $X\in \mathbb{R}^3$;
    \item[(ii)] $\Phi(X,0)=X$, $X\in\mathbb{R}^3$,
    $\mathcal{L}^3-a.e.$;
    \item[(iii)] For almost all $(X,t)\in\mathbb{R}^3\times(0,T)$
    \begin{equation}
\partial_t\Phi(X,t)=\tilde U(\Phi(X,t),t);
\end{equation}
    \item[(iv)] $\Phi(\cdot,t):\mathbb{R}^3\rightarrow\mathbb{R}^3$
    preserves   Lebesgue measure in
    $\mathbb{R}^3$ for all $t\in[0,T)$.
    \item[(v)]  We have
    \begin{equation}
    \Phi(X,t)\subset B(0,R(T))\;\mbox{ for almost all }\;(X,t)\in\nabla P_0(\Omega)\times [0,T).
    \end{equation}
     In particular,
    \begin{equation}
    \partial_t\Phi(X,t)=U(\Phi(X,t),t)\;\mbox{for almost all}\;(X,t)\in\nabla P_0(\Omega) \times[0,T).
    \end{equation}
    \item[(vi)] There exists a Borel mapping
    $\Phi^{\ast}:\mathbb{R}^3\times[0,T)\rightarrow\mathbb{R}^3$ such that,
    for all $t\in(0,T)$ the map
    $\Phi_t^{\ast}:\mathbb{R}^3\rightarrow\mathbb{R}^3$ preserves  Lebesgue
    measure in $\mathbb{R}^3$, and satisfies 
    $\Phi_t^{\ast}\circ\Phi_t(x)=x$ and
    $\Phi_t\circ\Phi_t^{\ast}(x)=x$, for almost all
    $x\in\mathbb{R}^3$,
    \item[(vii)] Under the conditions of Theorem 1.2, if $(\alpha, P)$
    is a weak solution of (\ref{du}), and if (i)- (vi) hold, then for any $t\in [0,T]$,
    \begin{equation}
    \alpha_t=\Phi_t\#\alpha_0.
    \end{equation}
    Moreover, for any $t\in[0,T],$
    \begin{equation}
\alpha_t(x)=\alpha_0(\Phi^{\ast}_t(x))\;\mbox{ for almost all }\;x\in\mathbb{R}^3.
\end{equation}
\end{enumerate}
\end{prop}

\section{Orlicz spaces}

In what follows, we will collect definitions and a few results
concerning Orlicz spaces. For details, see \cite{sob} and
\cite{helena}.

Consider $a:[0,\infty)\rightarrow[0,\infty)$ with the following properties:
\begin{enumerate}
    \item[(i)] $a(0)=0$, $a(t)>0$ if $t>0$ and
    $\displaystyle{\lim_{t\rightarrow\infty}a(t)=\infty}$,
    \item[(ii)] $a$ is non-decreasing,
    \item[(iii)] $a$ is right-continuous.
\end{enumerate}
The function $A$, defined on $[0,\infty)$ by taking
$$A(t)=\int_0^t a(\tau)d\tau,$$
is called an $N$-function.

We note that $N$-functions are continuous on $[0,\infty)$,
convex and strictly increasing.

An $N$-function is said to be $\Delta$-regular if there exists a positive constant $C$ and 
$t_0>0$ such that $A(2t)\leq C A(t),\;\forall t\geq t_0$.

Let $\Omega$ be a domain in $R^n$, and $A$ an $N$-function. The
Orlicz space $L_A(\Omega)$ is  the linear closure of the
set of functions $u\in L^1_{loc}(\Omega)$ such that $A(\vert u\vert)$
is integrable. These are Banach spaces with norm given by 
$$\Vert u\Vert_A=\inf\left\{k>0,\;\int_{\Omega}A\left(\frac{\vert
u(x)\vert}{k}\right)dx\leq1\right\}.$$ 

They generalize the Lebesgue spaces $L^p(\Omega)$, which are Orlicz spaces, with
$N$-function $A(t)=t^p$. We denote by $E_A(\Omega)$ the closure, with respect to $\|\cdot\|_A$,  
of the set of smooth, compactly supported functions in $\Omega$. For every
$N$-function $A$, we have that $E_A(\Omega)$ is separable. In
general, $L_A(\Omega)$ and $E_A(\Omega)$ are distinct, and $L_A(\Omega)$
is not separable. However, when $A$ is $\Delta$-regular, $L_A=E_A$.

Let $A$ be an $N$-function. Its Legendre transform 
$A^{\ast}$ is given  by
$$A^{\ast}=A^{\ast}(s)=\max_{t\geq 0}\{st-A(t)\}.$$
One can verify that $A^{\ast}$ is also an $N$-function and that
$A^{\ast\ast}=A$. 

Finally,  the following classical results will be relevant in our analysis.

\begin{teo}\cite{sob} \label{csob} 
The dual of $E_A(\Omega)$ is
$L_{A^{\ast}}(\Omega)$. \end{teo}

\begin{lema}\cite{chae} \label{cchae}
Let $\Omega$ be a bounded domain in
$R^n$ and $f$, $f^k \in L^1(\Omega)$, for all $k$. If $f^k \to f$ strongly in $L^1$
then there exists a $\Delta$-regular $N$-function $A$ such that
$\{f^k\}$ and $f$ are uniformly bounded in $L_A(\Omega)$.
\end{lema}

\begin{lema}\cite{helena} \label{chelena}
Let $\{u_n\}$ be a sequence of functions uniformly bounded in 
$L^{\infty}(\Omega)$. If $u_n\rightarrow u$ in
$L^1(\Omega)$, then $u_n\rightarrow u$ in $L_A(\Omega)$, for any $N$-function $A$.
\end{lema}

\section{ Weak stability of the semigeostrophic equations in physical space}

We now turn to the main objective of this article. We consider a sequence of initial
potential vorticities, converging strongly in $L^1$ to a given limit vorticity. We 
would like to understand the convergence properties of the corresponding Lagrangian
solutions in physical space. Our motivation for considering this problem was, originally, 
to try to extend Cullen and Feldman's construction to solutions of the semigeostrophic equations 
with measures as potential vorticities. To do so we intended to approximate such solutions 
by smoother ones and, hence, we needed to understand how the corresponding Lagrangian solutions 
behaved. As it turns out this approach to construct Lagrangian solutions for measure-valued 
potential vorticities does not work; this will be  made clear by means of a counterexample, in Section 5. 
Instead, we have established a weak stability, or continuity property, of Lagrangian
solutions with respect to integrable perturbations of an $L^1$ initial potential vorticity. Weak stability
of weak solutions of the semigeostrophic equations, in dual formulation, has already been established 
by G. Loeper in \cite{loeper}.
 
Throughout the remainder of this section we fix $\alpha_0 \in L^1(\real^3)$ together with a sequence $\{\alpha_0^n\} \subset L^1 (\real^3)$ such that 
$\alpha_0^n \to \alpha_0$ strongly in $L^1(\real^3)$. In addition, we assume that $\alpha_0$ and the sequence $\{\alpha_0^n\}$ are
compactly supported, with supports contained in a ball $B(0,R_0)$.

Using Lemma \ref{cchae}, as in \cite{helena}, we have that there exists a $\Delta$-regular $N$-function $A$ such that $\alpha_0^n$, $\alpha_0$ are uniformly bounded in $L_A(\real^3)$.

Let  $\alpha^n = \alpha^n(y,t)$ be a weak solution of the semigeostrophic equations in dual formulation with initial potential vorticity $\alpha_0^n$. Consider the corresponding modified pressures $P^n$, $P_0^n$, defined in the physical space    $\Omega$. Denote by  $\Phi^n$  the Lagrangian flow in dual space given in Proposition \ref{ambrosio}. Finally, consider the corresponding Lagrangian flows in physical space, 
\begin{equation}
F^n_t:=\nabla (P^{n}_t)^{\ast} \circ \Phi^n_t \circ\nabla P^n_0, 
\label{2} \end{equation} 
as obtained in Theorem \ref{theo}, equation \eqref{flu}.

In the proof of Theorem \ref{existheoMaut}, case $q=1$, (see \cite{helena}), it was shown that a subsequence of $\{\alpha^{n}\}$, $\{P^{n}\}$, 
$\{(P^{n})^{\ast}\}$ exists, which we do not re-label, together with a weak solution $\alpha=\alpha(x,t)$ of the semigeostrophic equations in dual formulation (with initial potential vorticity $\alpha_0$), such that the following hold, for each $t\in [0,T)$:
\begin{equation}
\begin{array}{ll}
\alpha^{n }(\cdot,t)\rightharpoonup\alpha(\cdot,t) & w^{\ast}-L_A(B(0,R(T)))\\
P^{n }(\cdot,t)\longrightarrow P(\cdot,t) & W^{1,r}(\Omega)\\
P^{n }_0\longrightarrow P_0 &W^{1,1}(\Omega)\\
(P^{n })^{\ast}(\cdot,t)\longrightarrow P^{\ast}(\cdot,t)& W^{1,1}_{loc}(\mathbb{R}^3)\\
\nabla (P^{n })^{\ast}(\cdot,t)\longrightarrow \nabla P^{\ast}(\cdot,t)&(E_{A^{\ast}})_{loc}(\mathbb{R}^3).
\end{array}
\label{1}
\end{equation}
In the proof of the convergence of $\nabla (P^{n })^{\ast}(\cdot,t)$ to $\nabla P^{\ast}(\cdot,t)$ one uses Theorem \ref{csob} and Lemma \ref{chelena}.

We fix, throughout the remainder of this section, such a subsequence.

Above, $P(\cdot,t)$, $P_0$ are the modified pressures corresponding to 
$\alpha(\cdot,t)$ and $\alpha_0$, and $R(T)$ is given in Theorem \ref{existheoMaut}, 
item (i). Let $\Phi$ be the Lagrangian flow in dual space as in Proposition \ref{ambrosio} and let 
\begin{equation} \label{limitF}
F_t:=\nabla P^{\ast}_t \circ \Phi_t \circ\nabla P_0,
\end{equation}
be the Lagrangian flow in physical space obtained in Theorem \ref{theo}, see \eqref{flu}.

Our main result is the following.

\begin{teo} \label{mainthm}
There exists a (further) subsequence $\{F^{n_k}_t\} \subset \{F^n_t\}$ such that, for almost every  $0\leq t<T$, we have:
\[F^{n_k}_t \longrightarrow F_t,\;\;\mbox{ strongly in }\;\; L^r (\Omega), \;\;\mbox{ as }\;\; k \to \infty,\]
for all $r\in [1,\infty)$.
\end{teo}

Before we present the proof of the theorem, we require the following 
auxiliary result.
 
\begin{lema} \label{Phin} For each $R>0$, we have 
$$\lim_{n\rightarrow\infty}\int_{B(0,R)}\sup_{t\in
[0,T]}\vert\Phi^n(X,t)-\Phi(X,t)\vert dX=0.$$
\end{lema}

\begin{proof}
For each $n$, recall that $\displaystyle{\Phi^n_t}$ is the Lagrangian flow
in dual space associated with the vector field
$\displaystyle{\tilde{U}^n(X,t)=J[H(X)-\nabla (P^{n}_t)^{\ast})]}$.
Since we have that $\displaystyle{(P^{n})^{\ast}(\cdot,t)\rightarrow P^{\ast}(\cdot,t)}$
in $W^{1,1}_{loc}(\mathbb{R}^3)$, it follows that 
\begin{equation}
\tilde{U}^n(X,t)\longrightarrow
\tilde{U}(X,t)\;\mbox{em}\;L^1_{loc}(\mathbb{R}^3). \label{3}
\end{equation}

However, such a condition is not enough to obtain 
the convergence of $\displaystyle{\{\Phi^n_t\}}$, due to the fact that we cannot control $\nabla \tilde{U}^n$. This is 
required in Ambrosio's stability result, namely Theorem 6.5 of \cite{amb}. 

We consider, instead, a family of aproximations of $\tilde{U}^n$ given by
\begin{equation}
\tilde{U}^{n,m}(X,t)=J[H(X)-(\nabla (P^{n}_t)^{\ast}\ast \eta^m)(X,t)],
\label{4}\end{equation} where $\eta^m$ is a standard mollifier.

Now we have: 
\begin{equation}
\begin{array}{l}
\tilde{U}^{n,m} \in C([0,T] \times \mathbb{R}^3,\mathbb{R}^3),\\
{\sup_{m} \Vert \tilde{U}^{n,m}\Vert_{L^{\infty}(\mathbb{R}^3\times[0,T],\mathbb{R}^3)}<\infty},\\
\mbox{div }\tilde{U}^{n,m}=0,\\ 
\Vert \nabla \tilde{U}^{n,m}\Vert_{L^{\infty}([0,T]\times B(0,R),\mathbb{R}^3)}\leq
C(n,m,R)<\infty,\\
\tilde{U}^{n,m}\longrightarrow \tilde{U}^{n}\;\mbox{in}\; L^1_{loc}(\mathbb{R}^3\times[0,T)).
\end{array}\label{5}
\end{equation}

Let $\Phi^{n,m}(X,t)$ be the Lagrangian flow associated to $\tilde{U}^{n,m}$. It follows from Theorem 6.5 in \cite{amb} that \begin{equation} \label{6}
\lim_{m\rightarrow\infty}\int_{B(0,R)} \; 
\sup_{t\in [0,T]}\vert\Phi^{n}(X,t)-\Phi^{n,m}(X,t)\vert dX=0,\;\forall R>0.
\end{equation}

Note  that 
$\tilde{U}^{n,m}(X,t)\rightarrow \tilde{U}(X,t)$ in $ L^1_{loc}(\mathbb{R}^3\times[0,T))$ when
$m,n\rightarrow \infty$. To see this it is enough to observe that, for any $R>0$, we have
$$\Vert \nabla (P^{n}_t)^{\ast}\ast \eta^m-\nabla
P^{\ast}_t\Vert_{L^1(B(0,R))} 
\leq\Vert \eta^m\Vert_{L^1(B(0,R))}
\Vert \nabla (P^n_t)^{\ast}-\nabla P^{\ast}_t\Vert_{L^1(B(0,R))} $$
$$+\Vert\nabla P^{\ast}_t\ast\eta^m-\nabla P^{\ast}_t\Vert_{L^1(B(0,R))}\stackrel{n,m\rightarrow\infty}{\longrightarrow} 0.$$

Hence, we also have
\begin{equation} \label{gap}
\lim_{m,n\rightarrow\infty}\int_{B(0,R)} \; \sup_{t\in
[0,T]}\vert\Phi(X,t)-\Phi^{n,m}(X,t)\vert dX=0,\;\forall R>0.
\end{equation}

Given \eqref{6}, it is possible to choose a subsequence $m=m(n) > n$ such that 

\begin{equation}
\lim_{n\rightarrow\infty}\int_{B(0,R)}\;
\sup_{t\in[0,T]}\vert\Phi^n(X,t)-\Phi^{n,m(n)}(X,t)\vert dX=0,\;\forall R>0.
\label{7}\end{equation}

We conclude, from \eqref{7} and \eqref{gap}, that
$$\lim_{n\rightarrow\infty}\int_{B_R}\sup_{t\in
[0,T]}\vert\Phi^n(X,t)-\Phi(X,t)\vert dX=0,\;\forall R>0,$$ which concludes the proof. 
 
\end{proof}

\begin{obs} Once we take into account the expression
(\ref{flu}),  Proposition 1.1 {\it{(v)}},
and (23), we see that we may assume in what follows that the flow $\Phi(X,t)$
is associated with the vector field $U(X,t)=J[X-\nabla P^{\ast}(X,t)]$.
\end{obs}

With this lemma we are now ready to give the proof of Theorem \ref{mainthm}.

\vspace{0.4cm}

\begin{proofmainthm} Let us first prove our result for the case $r=1$.

We note that, for each $0\leq t<T$, we have 
\begin{equation}
\begin{array}{l}
\displaystyle{\int_{\Omega}\vert F^n_t(x)-F_t(x)\vert 
\,dx}=\\
\displaystyle{\int_{\Omega}\vert\nabla ( P^n_t)^{\ast} \circ
\Phi^n_t \circ\nabla P^n_0(x)-\nabla P^{\ast}_t \circ \Phi_t
\circ\nabla P_0(x)\vert  \,dx}\\ 
\leq
\displaystyle{ \int_{\Omega}\vert\nabla ( P^n_t)^{\ast}
\circ \Phi^n_t \circ\nabla P^n_0(x)-\nabla (
P^n_t)^{\ast} \circ \Phi_t \circ\nabla P_0^n(x)\vert\, dx }\\
 +\displaystyle{\int_{\Omega}\vert\nabla ( P^n_t)^{\ast} \circ
\Phi_t \circ\nabla P^n_0(x)-\nabla P^{\ast}_t\circ \Phi_t
\circ\nabla
P_0^n(x)\vert\, dx} \\
+ \displaystyle{\int_{\Omega}\vert\nabla P^{\ast}_t \circ
\Phi_t \circ\nabla P_0^n(x)-\nabla P^{\ast}_t \circ \Phi_t
\circ\nabla
P_0(x)\vert\, dx} \\
\equiv  I_1+I_2+I_3 . \end{array}\end{equation}

We will show that each of these integrals vanish as $n$ tends to infinity, passing to subsequences as needed. 

Let us begin by considering $I_1$. Using that $\nabla P_0^n \sharp \chi_{\Omega} = \alpha_0^n$ we have:
\begin{equation}
\begin{array}{l}
\displaystyle{\int_{\Omega}\vert\nabla ( P^n_t)^{\ast} \circ \Phi^n_t
\circ\nabla P^n_0(x)-\nabla ( P^n_t)^{\ast} \circ \Phi_t \circ\nabla
P_0^n(x)\vert\,
dx}=\\\;\;\;\;\;\;\;\;\;\;\;\;\;\;\;\;\;=\displaystyle{\int_{\mathbb{R}^3}\vert\nabla (
P^n_t)^{\ast} \circ \Phi^n_t(y) -\nabla ( P^n_t)^{\ast} \circ
\Phi_t(y) \vert\, \alpha^n_0(y) dy}\\
\;\;\;\;\;\;\;\;\;\;\;\;\;\;\;\;\leq
 \displaystyle{ \int_{\mathbb{R}^3}\vert\nabla (
P^n_t)^{\ast} \circ \Phi^n_t(y) -\nabla P^{\ast}_t \circ
\Phi_t^n(y) \vert\, \alpha^n_0(y) dy }+\\
\;\;\;\;\;\;\;\;\;\;\;\;\;\;\;\;\;\;\;\;\;\;\;\;\;\;\;\;\;\;\;\;\;\;\;\;\;\;\;\;\; +\displaystyle{\int_{\mathbb{R}^3}\vert \nabla
P^{\ast}_t \circ \Phi^n_t(y) -\nabla ( P^n_t)^{\ast} \circ \Phi_t(y)
\vert\,
\alpha^n_0(y) dy} \\
\;\;\;\;\;\;\;\;\;\;\;\;\;\;\;\; \equiv I^1_1+I^2_1 .\end{array}
\end{equation}

From Proposition 1.1 (vii) we see that $\alpha^n_t=\Phi^n_t\#
\alpha_0^n$ and, since $L_A$ is a rearrangement invariant space, it follows that $\Vert
\alpha_0^n\Vert_{L_A}=\Vert\alpha_t^n\Vert_{L_A}$, for each $n$ and for each $t\in[0,T)$. Therefore, 
\begin{equation}
\begin{array}{ll}
I^1_1&=\displaystyle{\int_{\mathbb{R}^3}\vert\nabla ( P^n_t)^{\ast}
\circ \Phi^n_t(y) -\nabla P^{\ast}_t \circ \Phi_t^n(y) \vert\,
\alpha^n_0(y) dy}\\
&=\displaystyle{\int_{\mathbb{R}^3}\vert\nabla ( P^n_t)^{\ast} (z)
-\nabla P^{\ast}_t (z) \vert\, \alpha^n_t(z) dz}\\
&\leq\displaystyle{\Vert\nabla (P^n_t)^{\ast}  -\nabla P^{\ast}_t
\Vert_{E_{A^{\ast}}}\Vert\alpha^n_t\Vert_{L_A}}\\
&=\displaystyle{\Vert\nabla (P^n_t)^{\ast}  -\nabla P^{\ast}_t
\Vert_{E_{A^{\ast}}}\Vert\alpha^n_0\Vert_{L_A}}\stackrel{n\rightarrow\infty}{\longrightarrow}
0,
\end{array}
\end{equation}  where we have used  (\ref{1}), Theorem \ref{csob} and the boundedness of $\alpha_0^n$ in $L_A$.

As for $I^2_1$, we have:
\begin{equation}
\begin{array}{ll}
I^2_1&=\displaystyle{\int_{\mathbb{R}^3}\vert\nabla P^{\ast}_t
\circ \Phi^n_t(y) -\nabla ( P^n_t)^{\ast} \circ \Phi_t(y) \vert\,
\alpha^n_0(y) dy}\\
&\leq \displaystyle{\int_{\mathbb{R}^3}\vert\nabla
P^{\ast}_t \circ \Phi^n_t(y) -\nabla P^{\ast}_t \circ \Phi_t(y)
\vert\, \alpha^n_0(y)
dy} + \\
&\;\;\;\;\;\;\;\;\;\;\;\;\;\;\;\;\; +\displaystyle{\int_{\mathbb{R}^3}\vert\nabla
P^{\ast}_t\circ \Phi_t(y) -\nabla ( P^n_t)^{\ast} \circ \Phi_t(y)
\vert\,
\alpha^n_0(y) dy} \\
& \equiv I^{2,1}_1+I^{2,2}_1 .
\end{array}
\end{equation}

Consider the integral $I^{2,1}_1$. Since $\alpha_0^n \to \alpha_0$ strongly in $L^1$, it is easy to see that the proof that 
$I^{2,1}_1$ tends to zero as $n \to \infty$ reduces, by  Lebesgue's dominated
convergence theorem,  to showing that, for all $t$,
\begin{equation}\nabla P^{\ast}_t
\circ \Phi^n_t(y) -\nabla P^{\ast}_t \circ
\Phi_t(y)\longrightarrow
0,\;\;n\rightarrow\infty \;\;\mbox{ a.e. } y \in \real^3.\label{3.14}\end{equation}
At this point we must pass to a further subsequence. We have, by Lemma \ref{Phin}, that 
\[g^n := g^n(y) = \sup_{0<t<T} |\Phi^n(y,t) - \Phi(y,t)| \to 0 \mbox{ strongly in } L^1_{\loc}.\]
From this it follows that there exists a subsequence, $\{g^{n_k}\}$, which converges, a.e. $y \in \real^3$, to $0$ as $k \to \infty$.
Hence, for every $0\leq t<T$, we have 
\begin{equation} \label{graah}
\Phi^{n_k}_t(y)-\Phi_t(y)\rightarrow 0 \mbox{ a.e. } y \in \real^3 \mbox{ as } k \to \infty.
\end{equation}

Now, since $P^{\ast}_t$ is convex, it follows that $\nabla P^{\ast}_t$ 
is almost everywhere differentiable (see, for instance, \cite{evansgari}), and hence, continuous except for a set of Lebesgue 
measure zero, say $N \subset \real^3$. Given \eqref{graah} it is enough to show, therefore, that, for almost all $y \in \real^3$, $\nabla P^{\ast}_t$ is 
continuous at $\Phi_t(y)$. To see this we note that
$$\vert \{y\in \mathbb{R}^3;\;\Phi_t(y)\in
N\}\vert  =\vert \{y\in \mathbb{R}^3;\;y\in\Phi^{\ast}_t(N)\}\vert 
= \vert N \vert = 0,$$ 
in view of the fact that $\Phi^{\ast}_t$ preserves Lebesgue measure. Therefore, $\Phi_t(y)$ is a continuity point for
$\nabla P^{\ast}_t$, for almost all $y$, as desired. 

Hence,  
\begin{equation}\int_{\real^3}\vert\nabla P^{\ast}_t \circ \Phi^{n_k}_t(y)
-\nabla P^{\ast}_t \circ \Phi_t(y) \vert\, \alpha^{n_k}_0(y) \,dy\rightarrow 0 .\label{14}\end{equation}

The analysis of $I^{2,2}_1$ is similar. We have that $\nabla P^{\ast}_t(\mathbb{R}^3)$, 
$\nabla (P^n_t)^{\ast}(\mathbb{R}^3) \subset B(0,S),$ $\forall t,n$, $\nabla (P^n_t)^{\ast}\rightarrow \nabla P^{\ast}_t$ strongly in $L^1_{\loc}$, 
hence $\nabla (P^n)^{\ast} \to \nabla P^{\ast}$ strongly in $L^1_{\loc}(\real^3\times\real_+)$. Thus, we may pass to a subsequence, chosen independently of $t$, and which we do not re-label, so that
\[\nabla (P^{n_k}_t)^{\ast} \to \nabla P_t^{\ast}\;\mbox{ a.e. } y \in \real^3,\]
as $k \to \infty$, for almost every $0 \leq t < T$. 

Using this, together with the fact that $\Phi_t$ is measure preserving, we may conclude as before that  
$$\nabla ( P^{\ast}_t)^{n_k} \circ \Phi_t(y)-\nabla P^{\ast}_t \circ
\Phi_t(y)\longrightarrow 0, \mbox{ a.e. }y \in \real^3,$$
a.e. $0 \leq t < T$. This, together with the strong convergence in $L^1$ of $\alpha_0^n\rightarrow
\alpha_0$, and Lebesgue's dominated convergence theorem, yield,  
\begin{equation}\int_{\real^3}\vert\nabla ( P^{\ast}_t)^{n_k} \circ \Phi_t(y)
-\nabla P^{\ast}_t \circ \Phi_t(y) \vert\, \alpha^{n_k}_0(y)
dy\rightarrow 0 .\label{15}
\end{equation}
\vskip10pt 
From (\ref{14}) and (\ref{15}) we have that $I^2_1\rightarrow
0$, which concludes the analysis of $I_1$.

Next we consider $I_2$. Using the fact that $\nabla P_0^n\#
\chi_\Omega=\alpha_0^n$, we have that,
\begin{equation}
\begin{array}{l}
\displaystyle{\int_{\Omega}\vert\nabla ( P^n_t)^{\ast} \circ \Phi_t
\circ\nabla P^n_0(x)-\nabla P^{\ast}_t \circ \Phi_t \circ\nabla
P_0^n(x)\vert\, dx}=\\
\;\;\;\;\;\;\;\;\;\;\;\;\;\;=\displaystyle{\int_{\mathbb{R}^3}\vert\nabla (
P^n_t)^{\ast} \circ \Phi_t (y)-\nabla P^{\ast}_t \circ \Phi_t
(y)\vert\,\alpha_0^n(y) dy}
\end{array}
\end{equation}
which is the same as $I^{2,2}_1$. Hence, from (\ref{15}), it follows that, passing to the appropriate subsequence, $I_2\rightarrow 0$. 

Finally, we consider the last integral,  
$$\begin{array}{ll}
I_3&=\displaystyle{\int_{\Omega}\vert\nabla P^{\ast}_t \circ
\Phi_t \circ\nabla P_0^n(x)-\nabla P^{\ast}_t \circ \Phi_t
\circ\nabla P_0(x)\vert\, dx}.
\end{array}$$
Now, since $\nabla P^{\ast}$ is bounded, it is enough, by the Lebesgue dominated convergence theorem, to prove that 
\begin{equation}\nabla P^{\ast}_t
\circ \Phi_t \circ\nabla P_0^{n_k}(x)-\nabla P^{\ast}_t \circ \Phi_t
\circ\nabla P_0(x)\rightarrow 0 \mbox{ a.e. } x \in \Omega.
\label{44}
\end{equation}
To this end we, once more, pass to a subsequence for which $\nabla P_0^{n_k} \to \nabla P_0$ a.e. $x \in \Omega$, and we do not further re-label. 

Recall that the support of $\alpha_0^{n}$ was assumed to be contained in the ball $B(0,R_0)$, for all $n$, and that $\nabla P_0^n (\Omega)$ is 
precisely the support of $\alpha_0^n$, hence contained in $B(0,R_0)$. 

Next, we note that, passing to subsequences as needed, $\Phi_t \circ \nabla P_0^{n_k} \to \Phi_t \circ \nabla P_0$ a.e. $x \in \Omega$. We will show this by proving the convergence of $\Phi_t \circ \nabla P_0^{n_k} \to \Phi_t \circ \nabla P_0$ in $L^1$ and passing to a subsequence which converges a.e. $x \in \Omega$.

By Lusin's theorem we have that $\Phi_t$ coincides with a continuous function up to a set of arbitrarily small Lebesgue measure. More precisely, let 
$\vare > 0$ and consider $f^{\vare} \in C^0(B(0,R_0))$ and $E^{\vare} \subset B(0,R_0)$ such that $\Phi_t = f^{\vare}$ outside of $E^{\vare}$ and 
$|E^{\vare}| < \vare$. Since $\Phi_t$ is bounded, for each $t$, we may assume that $f^{\vare}$ is also bounded, uniformly in $\vare$. We use the fact that $\nabla P_0^{n_k} \# \chi_{\Omega}=\alpha_0^{n_k}$, and the analogous fact for $P_0$, to estimate:
\[\limsup_{n_k \to \infty} \int_{\Omega} |\Phi_t \circ \nabla P_0^{n_k} - \Phi_t \circ \nabla P_0|\, dx \leq 
\limsup_{n_k \to \infty} \int_{\Omega} |(\Phi_t  - f^{\vare}) \circ \nabla P_0^{n_k}|\, dx 
\]
\[+\limsup_{n_k \to \infty} \int_{\Omega} |f^{\vare} \circ \nabla P_0^{n_k} - f^{\vare} \circ \nabla P_0|\, dx +  
  \int_{\Omega} |(f^{\vare} - \Phi_t) \circ \nabla P_0|\, dx \]
\[
=\limsup_{n_k \to \infty} \int_{E^{\vare}} |\Phi_t  - f^{\vare}|\, d\alpha_0^{n_k} + \limsup_{n_k \to \infty} \int_{\Omega} |f^{\vare} \circ \nabla P_0^{n_k} - f^{\vare} \circ \nabla P_0|\, dx + \int_{E^{\vare}} |f^{\vare} - \Phi_t|\, d\alpha_0 
\]
\[
\leq 2   \|\Phi_t - f^{\vare}\|_{L^{\infty}}(|\alpha_0^{n_k}(E^{\vare})| + |\alpha_0(E^{\vare})|) + \limsup_{n_k \to \infty} \int_{\Omega} |f^{\vare} \circ \nabla P_0^{n_k} - f^{\vare} \circ \nabla P_0|\, dx.
\]
The first term can be made arbitrarily small since $\{\alpha_0^n\}$ is uniformly integrable, while the second term vanishes because $f^{\vare}$ is continuous. 

We have shown that $\Phi_t \circ \nabla P_0^{n_k} \to \Phi_t \circ \nabla P_0$ a.e. $x \in \Omega$, passing to a further subsequence if needed. Next, recall that $\nabla P^{\ast}_t$ is continuous in $\mathbb{R}^3\setminus N$, so that, passing to the subsequence above, to obtain (\ref{44}) it is enough to show that
$$\vert\{x\in \Omega;\;\Phi_t\circ\nabla P_0(x)\in N\}\vert=0.$$
Recall that $\Phi^{\ast}_t$ preserves Lebesgue measure, so that $|\Phi^{\ast}_t(N)| = |N| = 0$. With this we obtain, using again that 
$\nabla P_0 \# \chi_{\Omega} = \alpha_0$, 
\begin{equation}
\begin{array}{ll}
\vert\{x\in \Omega;\;\Phi_t\circ\nabla P_0(x)\in
N\}\vert&=\vert\{x\in \Omega;\;\nabla P_0(x)\in
\Phi^{\ast}_t(N)\}\vert\\
& = \displaystyle{\int_{\Omega} \chi_{\Phi^{\ast}_t(N)} \circ \nabla P_0 (x) \,dx }\\
&=\displaystyle{\int_{\Phi^{\ast}_t(N)}\alpha_0(y)\,dy}=0,
\end{array}
\end{equation}
and therefore, $I_3\rightarrow0$. 

This establishes our result if $r=1$. 

Now, given that $\nabla (P^{n}_t)^{\ast}$ is uniformly bounded, we obtain the convergence in $L^r(\Omega)$, $1<r<\infty$, by interpolation. 
This concludes the proof.

\end{proofmainthm}

\section{Weak stability for the shallow water case}

The shallow water version of the semigeostrophic equations can
be written as an equation for $h=h(x,t)$, $x = (x_1,x_2) \in \Omega \subset \real^2$, $t \in [0,T)$, and 
$v = (v_1,v_2)$. Here, $h$ is the height of
fluid above  $\Omega $ and $v$ is the velocity. 
We denote $D_t = \partial_t + v \cdot \nabla$ and we set
\[P = P(x ,t)=h(x ,t)+\frac{1}{2}|x|^2 \;\;\;\mbox{ and } X = \nabla P.\]

With this notation, the shallow water SG equations take the form:
 
\begin{equation}
\left\{ \begin{array}{ll}
D_t X=(X-x)^{\perp},&\;\;\;\Omega\times (0,T),\\
\partial_th+ \mbox { div}(hv)=0,&\;\;\;\Omega\times(0,T),\\
v\cdot\nu=0,&\;\;\;\mbox{on }\;\partial\Omega\times[0,T),\\
h(x,0)=h_0(x),&\;\;\;\mbox{in }\;\Omega. \end{array} \right. \label{larasa}
\end{equation}
Here, $(a,b)^{\perp} = (-b,a)$.

In dual variables, this problem can be written as

\begin{equation} 
\left\{\begin{array}{ll}
\partial_t\alpha+\nabla \cdot (U\alpha)=0,&\mathbb{R}^2\times(0,T)\\
\nabla P_t\#h_t=\alpha_t;& t\in(0,T)\\
U(X,t)=[X-\nabla P^{\ast}(X,t)]^{\perp}&\mathbb{R}^2\times[0,T)\\
P^{\ast}(X,t)=\displaystyle{\sup_{x\in\Omega}\{x\cdot X-P(x,t)\}},&\mathbb{R}^2\times[0,T)\\
\alpha(\cdot,0)=\alpha_0 \equiv \nabla P_0 \# h_0.
\end{array}\right.\label{durasa}
\end{equation}

For the modeling background concerning this system, see \cite{culefel,culegan}.
A weak solution for this system was obtained by M. Cullen and W. Gangbo,
see \cite{culegan}, in the case $p>1$, and their existence result is similar to 
Theorem 1.1. In \cite{culefel}, Cullen and Feldman also proved existence of Lagrangian
solutions in physical space for the system \eqref{larasa} for $p>1$. The existence
results, both for weak solutions in dual variables (from \cite{culegan}) and for
Lagrangian solutions in physical variables (from \cite{culefel}), can be extended to 
$p=1$. The proof for weak solutions is an easy adaptation of the work in \cite{helena}, whereas
the proof for Lagrangian solutions is, as before, embedded in what follows. 
 
We are interested in Lagrangian solutions
$(P,F)$, where $\displaystyle{F:\Omega\times[0,T)\rightarrow\Omega}$ 
is a Lagrangian flow associated with $v$, and $P$ is obtained from a weak 
solution in dual variables. However, for the shallow water case, the vector field $v$ is not
divergence-free. Nevertheless, the transport equation $\partial_t h+ \mbox{ div}(hv)=0$ holds.
Therefore, if $F$ is a Lagrangian flow associated with $v$, the solutions
$h$ of this equation satisfy $F_t\#h_0=h_t,\;\forall
t\in[0,T)$. This property replaces the fact that $F$ preserves the Lebesgue 
measure in the incompressible case.

Let $\Omega\subset\mathbb{R}^2$ be open and bounded and let $T>0$. Let 
$P_0 = P_0(x)$ be a convex, bounded function in $\Omega$ such that 
$\displaystyle{h_0(x)=P_0(x)-\frac{1}{2}\vert x\vert^2\geq0}$ in 
$\Omega$. Let $r\in[1,\infty)$ and $P:\Omega\times[0,T)\rightarrow
\mathbb{R}$ be such that 
\begin{eqnarray}
P\in L^{\infty}([0,T),W^{1,\infty}(\Omega))\cap
C([0,T),W^{1,r}(\Omega))\\
P(\cdot ,t)\;\mbox{ is convex in }\;\Omega\;\mbox{ for each }\;t\in
[0,T).\end{eqnarray} Let 
$\displaystyle{h(x,t)=P(x,t)-\frac{1}{2}\vert x\vert^2}$. Let 
$F:\Omega\times[0,T)\rightarrow \Omega$ be a  Borel map satisfying 

\begin{equation}
F\in C([0,T),L^r(\Omega,h_0dx)).
\end{equation} 

\begin{defi}(Lagrangian Solutions)
The pair $(P,F)$ is called a weak Lagrangian solution of \eqref{larasa} in 
$\Omega\times[0,T)$ if 
\begin{enumerate}
    \item[(i)] $F(x,0)=x$, $h_0$-a.e. in $\Omega$, $P(x,0)=P_0(x)$ a.e. in $\Omega$,
    \item[(ii)] for every $t>0$ the map 
    $F_t=F(\cdot ,t):\Omega\rightarrow\Omega$ is such that $F_t\#h_0=h_t$,
    \item[(iii)] There exists a Borel map 
    $F^{\ast}:\Omega\times[0,T)\rightarrow\Omega$ such that, for each 
    $t\in (0,T)$ we have 
    $F^{\ast}_t=F^{\ast}(\cdot ,t)=\Omega\rightarrow\Omega$ satisfies $F^{\ast}_t\#h_t=h_0$, and 
    $F_t\circ F^{\ast}_t(x)=x\;h_t-a.e.\;\mbox{ in }\;\Omega$ and  $F^{\ast}_t\circ
    F_t(x)=x\;h_0-a.e.\;\mbox{ in }\;\Omega$,
    \item[(iv)] The function 
    \begin{equation}
    Z(x,t)=\nabla P(F_t(x),t)\label{2.36rasa}
    \end{equation}
    is a weak solution of 
    \begin{equation}
    \begin{array}{ll}
    \partial_tZ(x,t)=[Z(x,t)-F(x,t)]^{\perp},& \mbox{ on }\;supp\;h_0\;\mbox{ in }\quad\Omega\times[0,T)\\
    Z(x,0)=\nabla P_0(x),&\mbox{ on }\;supp\;h_0\;\mbox{ in }\quad\Omega,
    \end{array}\label{2.37rasa}\end{equation}
    in the following sense: for every $\varphi\in
    C^1_c(\Omega\times[0,T)),$
    \begin{equation}
    \begin{array}{l}
    \displaystyle{\int_{\Omega\times[0,T)}[Z(x,t)\cdot \partial_t\varphi(x,t)+(Z(x,t)-F(x,t))^{\perp}\cdot \varphi(x,t)]h_0(x)dxdt}+\\
    \;\;\;\;\;\;\;\;\;\;\;\;\;\;\;\;\;\;\;\;\;\;\;\;\displaystyle{+\int_{\Omega}\nabla
    P_0(x)\cdot\varphi(x,0)h_0(x)dx=0.}
    \end{array}\label{2.38rasa}
    \end{equation}
\end{enumerate}
\end{defi}
The map $F_t$ is a Lagrangian flow in physical space.
 
With this definition in place we give the precise statement of the existence of Lagrangian solutions in the shallow water case. 

\begin{teo} Let  $\Omega\subset\mathbb{R}^2$ be open and bounded, and assume that 
$\overline\Omega\subset B$, where $B$ is the open ball $B(0,S)$.
Let  $h_0(x)\geq 0$ be such that $P_0 = P_0(x)=h_0(x)+\frac{1}{2}\vert
x\vert^2$ is a convex, bounded function in $B$, and assume that 
\begin{equation}
DP_0\#h_0\in L^q(\nabla P_0(\Omega)) \end{equation} for some  
$q\geq 1$. Then, for each  $T>0$, there exists a Lagrangian solution 
$(P,F)$ of (\ref{larasa}) in $\Omega\times[0,T)$, where (50)--(52)
hold for all $\displaystyle{r\in[1,\infty)}$. Furthermore, the function $Z=Z(x,t)$ defined in (\ref{2.36rasa}) satisfies
$Z(x,\cdot )\in W^{1,\infty}([0,T))$ $h_0$-almost everywhere in $\Omega$, and 
(\ref{2.37rasa}) is also satisfied in the following sense
\begin{equation}
\begin{array}{ll}
\partial_tZ(x,t)=(Z(x,t)-F(x,t))^{\perp},&h_0\mathcal{L}^2\times\mathcal{L}^1\;\mbox{ a.e. in }\;\Omega\times(0,T),\\
Z(x,0)=\nabla P_0(x),&h_0\mathcal{L}^2\;\mbox{ a.e. in }\;\Omega.
\end{array}
\end{equation}
\end{teo}

As before, one obtains  $P$ from the dual problem; Cullen and Feldman showed that $F(\cdot,t)$, given by the expression below, is a Lagrangian flow in physical space:
$$F(x,t)=\nabla P^{\ast}_t\circ\Phi_t\circ\nabla P_0(x),$$
where, for each  $t$, $\Phi_t(X)$ is the Lagrangian flow in dual space associated to the vector field  $U(X,t)=[H(X)-\nabla P^{\ast}_t(X)]^{\perp}$,
whose construction in $\mathbb{R}^2$ arises in the same manner as for the incompressible case. 

Let us now address the stability of Lagrangian solutions. Consider $\alpha_0, \alpha_0^n \in L^1(\mathbb{R}^2)$
with $\alpha_0^n \to \alpha_0$ in $L^1$ and with supports contained in a single ball $B(0,R_0)$. Let $\alpha^n = \alpha^n(x,t)$
be weak solutions in dual variables with initial data $\alpha_0^n$. As before, there exists an Orlicz space $L_A$, with
$\Delta$-regular $N$-function $A$, such that $\{\alpha_0^n\}$, $\alpha_0$ is uniformly bounded in $L_A(\real^3)$.

Let $\alpha^n = \alpha^n(x,t)$ be a weak solution of \eqref{durasa} with initial data $\alpha_0^n$ and let $h^n$ be the corresponding height and $P^n$ be the corresponding modified pressure. It can be easily deduced, from the proofs of Lemma 3.6, Lemma 4.3 and Theorem 4.4 of \cite{culegan} that, since $\alpha_0^n \to \alpha_0$ strongly in $L^1$, there exists a subsequence such that: 
\begin{equation}
\begin{array}{l}
\alpha^n(\cdot,t) \rightharpoonup \alpha(\cdot,t)\;\mbox{ weak}-\ast\; \mathcal{BM},\\
h^n(\cdot,t)\longrightarrow h(\cdot,t)\;\mbox{ in }\;L^{\infty}(\Omega),\\
(P^n_t)^{\ast} \longrightarrow P_t^{\ast}\;\mbox{ strongly in } W^{1,1}_{\loc},\\
\nabla (P^n_t)^{\ast} \rightharpoonup \nabla P_t^{\ast}\;\mbox{ weak}-{\ast} L^{\infty}, 
\end{array}
\label{h}
\end{equation} for each  $0 \leq t < T$, with $\alpha$, $h$, $P^{\ast}$ a weak solution of the semigeostrophic shallow water equations. 
  
We use $\Phi^n$ to denote the Lagrangian flow in the dual space associated to $U^n$, and we denote by 
$F^n_t:=\nabla ( P^n_t)^{\ast}\circ\Phi^n_t\circ\nabla P_0^n$ the corresponding 
Lagrangian flow in physical space.  Accordingly, let $\Phi$ denote the Lagrangian flow in dual variables, associated to the limit velocity $U$ and let 
$F_t:= \nabla P_t^{\ast} \circ \Phi_t\circ \nabla P_0$ be the corresponding Lagrangian flow in physical space.

We note that the result in Lemma 3.1  remains valid in the present case. 

\begin{teo}  There exists a subsequence $\{F^{n_k}_t\} \subset \{F^n_t\}$ such that, for almost every $t \in [0,T)$, we have  
\begin{equation}
\displaystyle{\lim_{k\rightarrow\infty}\int_{\Omega}\vert F^{n_k}_t(x)-F_t(x)\vert^r \, h_0(x)dx=0;} \end{equation}
for any $r\in[1,\infty)$.
\end{teo}

\begin{proof}
Since $F^n$ is bounded uniformly in $L^{\infty}([0,T)\times\Omega)$, since $h_0^n \to h_0$ uniformly, and since $\Omega$ is bounded, it is clearly enough to prove that:
\begin{equation} \label{shwatermainthm}
\lim_{n\rightarrow\infty}\int_{\Omega}\vert
F^n_t(x)-F_t(x)\vert^r h^n_0(x)dx=0.
\end{equation}

To show \eqref{shwatermainthm} we note that, as in the incompressible case, we need only analyze the case $r=1$, as $r>1$ follows by interpolation. 
We  have:
\begin{equation}
\begin{array}{l}
\displaystyle{\int_{\Omega}\vert F^n_t(x)-F_t(x)\vert
h_0^n(x)dx}\leq\\
\;\;\;\;\;\;\;\;\;\;\;\;\leq\displaystyle{\left\{\int_{\Omega}\vert\nabla (
P^n_t)^{\ast} \circ \Phi^n_t \circ\nabla P^n_0(x)-\nabla (
P^n_t)^{\ast} \circ \Phi_t \circ\nabla P_0^n(x)\vert h_0^n(x)dx\right.}\\
\;\;\;\;\;\;\;\;\;\;\;\;\left.+\displaystyle{\int_{\Omega}\vert\nabla (
P^n_t)^{\ast} \circ \Phi_t \circ\nabla P^n_0(x)-
\nabla (P^{\ast}_t)\circ \Phi_t \circ\nabla
P_0^n(x)\vert h_0^n(x)dx}\right.\\
\;\;\;\;\;\;\;\;\;\;\;\;+\left.\displaystyle{\int_{\Omega}\vert\nabla
P^{\ast}_t \circ \Phi_t \circ\nabla P_0^n(x)-\nabla P^{\ast}_t
\circ \Phi_t \circ\nabla
P_0(x)\vert h_0^n(x)dx}\right\}\\
\;\;\;\;\;\;\;\;\;\;\;\;\;\equiv\{\tilde I_1+\tilde I_2+\tilde I_3\}.
\end{array}\end{equation}

The analysis of each of these integrals follows closely the analysis performed on the analogous integrals in Proposition 3.1, once we use the facts   that $\nabla P_0^n\#h_0^n=\alpha_0^n,\;\Phi^n_t\#\alpha_0^n=\alpha^n_t$.     
\end{proof}

\section{An example in the space of measures}

The purpose of this section is to describe a counterexample for Theorem \ref{mainthm} for potential vorticities which are
not absolutely continuous with respect to the Lebesgue measure.

For the discussion in this section we will ignore the vertical variable in the incompressible 
SG equations; the argument we will present can be easily adapted to accomodate the third direction.
 
We fix the physical space to be the planar disk $\Omega=\displaystyle{{B(0,1)}}$. Let $z_0 = (1,0)$ and set 
\[\alpha_0 = \pi \delta_{z_0},\]
where $\delta_P$ denotes the Dirac measure at $P$.
Let $\alpha(t) = \pi \delta_{z(t)}$, with $z(t) = (\cos t, \sin t)$. It can be checked that $\alpha$ is
a weak solution of \eqref{du}, in the sense of \cite{loeper}. 
Next, observe that the unique (up to a constant) convex potential for the optimal transport map between $\chi_{\Omega}$ and $\alpha(t)$
is given by $P= P(x,t)=z(t) \cdot x$. 
Its Legendre transform is $P^{\ast}(y,t)=\Vert y-z(t)\Vert$ and, consequently,   
$\nabla P^{\ast}(y,t)=\frac{y-z(t)}{\Vert y-z(t)\Vert}$. 
The Lagrangian flow in dual space, restricted to the support of $\alpha_0$, is precisely $z_0 \mapsto z(t)$. 
The Lagrangian flow in physical space cannot be computed by \eqref{flu}, since $\Phi_t \circ \nabla P_0(\cdot)$
is identically equal to $z(t)$, where $\nabla P^{\ast}$ is not defined. 

One can use approximations as a strategy to circumvent the difficulty described above; as we will show,
this does not work.

\begin{prop}
Let $z(t) = (\cos t,\sin t)$ and set
\[\alpha^{\vare} \equiv \frac{1}{\vare^2}\chi_{B(z(t),\varepsilon)}.\]
Then $\alpha^{\vare}$ is an exact weak solution of the semigeostrophic equations in dual variables
\eqref{du} with initial potential vorticity $\alpha^{\vare}(\cdot,0)$.
\end{prop}

\begin{proof}

Let us first establish the relation between a potential vorticity of the form $\alpha^{\vare}$
and the corresponding velocity $U^{\vare}$. To this end, we fix $\overline{z} \in \real^2$ and we consider
\[\overline{\alpha}^{\vare}\equiv \frac{1}{\vare^2}\chi_{B(\overline{z},\varepsilon)}.\]  
The optimal transport map between $\chi_{\Omega}$ and $\overline{\alpha}^{\vare}$ 
is given by $\nabla \overline{P}^{\vare}$, where the convex potential $\overline{P}^{\vare}$ is, up to a constant, 
\[\overline{P}^{\vare}= \overline{P}^{\vare}(x) = \overline{z}\cdot x+\varepsilon \frac{\vert x\vert^2}{2},\]
and hence $\nabla \overline{P}^{\vare}(x)=\overline{z}+\varepsilon x$. Indeed, it can be easily verified that 
$\nabla \overline{P}^{\vare} \# \chi_{\Omega} = \overline{\alpha}^{\vare}$ and $\overline{P}^{\vare}$ is convex, so that the 
uniqueness part of Brenier's Polar Factorization Theorem, see \cite{bre}, may be applied. The Legendre transform of $\overline{P}^{\vare}$ is
\[
(\overline{P}^{\vare})^{\ast} = (\overline{P}^{\vare})^{\ast}(y) = \left\{\begin{array}{ll}
\displaystyle{\frac{\Vert y - \overline{z}\Vert^2}{2\vare}}, &\mbox{ if } y \in B(\overline{z},\vare)\\ \\
\Vert y - \overline{z}\Vert - \frac{\vare}{2}, &\mbox{ if } y \notin B(\overline{z},\vare). 
\end{array} \right. \]
Therefore, we find that 
\begin{equation} \label{nablaPast}
\nabla(\overline{P}^{\vare})^{\ast} = \nabla(\overline{P}^{\vare})^{\ast}(y) = \left\{\begin{array}{ll}
\displaystyle{\frac{y - \overline{z}}{\vare}}, &\mbox{ if } y \in B(\overline{z},\vare)\\ \\
\displaystyle{\frac{y-\overline{z}}{\Vert y - \overline{z}\Vert}}, &\mbox{ if } y \notin B(\overline{z},\vare). 
\end{array} \right.
\end{equation} 

For each fixed $t$, we have that $\alpha^{\vare}$ is of the form $\overline{\alpha}^{\vare}$ with $\overline{z}=z(t)$. Therefore the corresponding semigeostrophic velocity $U^{\vare}$, in dual variables, is given by
\[U^{\vare}=U^{\vare}(y,t) = (y - \nabla (P^{\vare}_t)^{\ast}(y))^{\perp},\]
where $\nabla (P^{\vare}_t)^{\ast}$ is given by the expression in \eqref{nablaPast} with $\overline{z}=z(t)$.

Consider $ y_0 \in B(z_0,\vare)$. Let $y=y(t)$ be the solution of 
\[\left\{
\begin{array}{l}
y^{\prime} = U^{\vare}(y,t),\\
y(0)=y_0.
\end{array} 
\right.
\]
As long as $y(t) \in B(z(t),\vare)$ we see that 
\[ 
y^{\prime} = \left( 
\frac{\vare -1}{\vare}y + \frac{z(t)}{\vare}
\right)^{\perp}.\\
\]
We also have
\[z^{\prime} = z^{\perp}.\]
Thus, subtracting these two equations, we deduce that
\[\left\{
\begin{array}{l}
(y-z)^{\prime} =  \displaystyle{\frac{\vare -1}{\vare}}(y-z)^{\perp},\\ \\
(y-z)(0)=y_0 - z_0 \in B(0,\vare).
\end{array} 
\right.
\]
Therefore $y-z$ rotates around the origin at the rate $(\vare -1)/\vare$ hence, in particular, $y(t)$ rotates around $z(t)$ and never 
leaves $B(z(t),\vare)$. We have shown that the flow of $U^{\vare}$ maps $B(z_0,\vare)$ to $B(z(t),\vare)$ through a rigid rotation.

This implies that $\alpha^{\vare}$ is a weak solution of the transport equation $\partial_t \alpha^{\vare} + U^{\vare}\cdot\nabla\alpha^{\vare} = 0$, as 
desired.
 
\end{proof} 

\begin{obs} Note that $\alpha^{\vare}(\cdot,t) \rightharpoonup \alpha(t)$ weak-$\ast$ $\mathcal{BM}$, in accordance with \cite{loeper}. 
\end{obs}

\begin{obs} From the proof above we obtain an explicit expression for the Lagrangian flow in dual variables for Lagrangian markers inside $B(z_0,\vare)$, namely:
\begin{equation} \label{lafdu}
\Phi^{\vare}=\Phi^{\vare}_t(y_0)= z(t) + \left[
\begin{array}{lr}
\cos \left(\frac{\vare -1}{\vare} t \right) & -\sin \left(\frac{\vare -1}{\vare} t \right) \\ \\
\sin \left(\frac{\vare -1}{\vare} t \right) & \cos \left(\frac{\vare -1}{\vare} t \right)
\end{array} 
\right] (y_0 - z_0).
\end{equation} 
\end{obs}

Next we compute the Lagrangian flow in physical space associated to $\alpha^{\vare}$ using expression \eqref{flu}. Let $x \in \Omega$ and note that 
$\nabla P_0^{\vare}(x) = z_0 + \vare x \in B(z_0,\vare)$. Hence we may use the Lagrangian map \eqref{lafdu},  
together with the expression  in \eqref{nablaPast} with $\overline{z}=z(t)$, to obtain:
\begin{equation} \label{Fvare}
F^{\vare}_t= F^{\vare}_t(x)= \left[
\begin{array}{lr}
\cos \left(\frac{\vare -1}{\vare} t \right) & -\sin \left(\frac{\vare -1}{\vare} t \right) \\ \\
\sin \left(\frac{\vare -1}{\vare} t \right) & \cos \left(\frac{\vare -1}{\vare} t \right)
\end{array} 
\right]x.
\end{equation} 

In other words, as $\vare \to 0$, $F_t^{\vare}$ describes a rotation around the origin in physical 
space with arbitrarily large angular velocity. 
 In short, a concentrated vortex in dual space corresponds to a Lagrangian fast eddy in physical
space, but concentrating the dual space vortex into a point produces an unphysical eddy which 
rotates at infinite speed. This shows that it is impossible to extend the weak stability theory we 
developed here in $L^1$ to the full space of measures, while keeping the strong convergence of sequences 
of Lagrangian flows as a conclusion. There are two possibilities for further work in this direction. One 
is to develop a theory of {\it weak} convergence of Lagrangian flows associated with converging 
sequences of potential vorticities in the space of measures and another is to try to extend 
the $L^1$ theory  to spaces of continuous measures, considering that Diracs in dual 
space are associated with unphysical infinite velocity eddies and are, therefore, unphysical themselves,
but perhaps other measures, such as potential vortex sheets, may be associated with meaningful flows.   

We conclude with the following remark. We established the convergence of Lagrangian flows  a.e. in time,
$L^r$ in space. However, this may not be optimal, and this leads to an interesting line of investigation. It was pointed 
out, by Brenier and Gangbo in \cite{bregan}, that the topology induced by $L^r$-convergence in the space of diffeomorphisms is not very satisfactory. One may investigate, for instance, whether the convergence of Lagrangian flows can be improved for potential vorticities in H\"{o}lder spaces, using the regularity theory for optimal transport developed by Ma, Trudinger and Wang in \cite{matrudwang}. 

\vspace{0.5cm} 

{\small Acknowledgments: The research presented here is part of the PhD thesis of J. C. O. Faria, who 
was supported in part by CNPq grant \#141.217/2004-9. The research of M. 
C. Lopes Filho is supported in part by CNPq grant \#303.301/2007-4 and the research of 
H. J. Nussenzveig Lopes is supported in part by CNPq grant \#302.214/2004-6. This work acknowledges 
the support of FAPESP grant \#2007/51490-7.}

\end{document}